\documentclass{amsart} 
\usepackage[english]{babel} 

\usepackage[utf8]{inputenc}

\usepackage{amsmath}
\usepackage{amsthm}
\usepackage{amsfonts}
\usepackage{indentfirst}
\usepackage{graphicx}
\usepackage{xcolor}

\usepackage{amssymb}
\usepackage[mathscr]{eucal}

\newtheorem{theorem}{Theorem}[section]
\newtheorem{lemma}[theorem]{Lemma}
\newtheorem{corollary}[theorem]{Corollary}
\newtheorem{proposition}[theorem]{Proposition}

\theoremstyle{definition}

\newtheorem{remark}[theorem]{Remark}

\newcommand{\ctM}{\Theta_{M}}
\newcommand{\cts}{\Theta_{S_{\delta}}}

\newcommand{\cs}{S_{\delta}}

\newcommand{\PP}{\mathbb{P}}
\newcommand{\CC}{\mathbb{C}}
\newcommand{\ZZ}{\mathbb{Z}}
\newcommand{\mcL}{\mathcal{L}}
\newcommand{\mcO}{\mathcal{O}}

\newcommand{\mcH}{\mathcal{H}}
\newcommand{\mcK}{\mathcal{K}}
\newcommand{\mcE}{\mathcal{E}}
\newcommand{\mcF}{\mathcal{F}}

\begin{document}

 \begin{center}


 \end{center}
\title {Foliations with isolated singularities on Hirzebruch surfaces}
\author{C. Galindo}
\address{Institut Universitari de Matemàtiques i Aplicacions de Castelló (IMAC) and
         Departament de Matemàtiques, Universitat Jaume I, Edifici TI (ESTCE),
         Av. de Vicent Sos Baynat, s/n, Campus del Riu Sec, 12071 Castelló de la Plana, Spain}
\email{galindo@uji.es}
\author{F. Monserrat}
\address{Instituto Universitario de Matemática Pura y Aplicada,
         Universidad Politécnica de Valencia, Edificio 8E, acceso F, 4a Planta,
         Camino de Vera, s/n, 46022 Valencia, Spain}
\email{framonde@mat.upv.es}
\author{J. Olivares}
\address{Centro de Investigaci\'on en Matem\'aticas, A.C.
A.P. 402, Guanajuato 36000, Mexico.} \email{olivares@cimat.mx}
\thanks{The first two authors are partially supported by the Spanish Government
MICINN/FEDER/AEI/UE, grants  PGC2018-096446-B-C22 and  RED2018-102583-T, as well
as by Generalitat Valenciana, grant AICO-2019-223 and Universitat Jaume I, grant UJI-2018-10.
The third author was partially supported by CONACYT: Estancias Sabáticas Vinculadas a la
Consolidación de Grupos de Investigación, CVU 10069}

\begin{abstract}
We study foliations $\mathcal{F}$ on Hirzebruch surfaces $S_\delta$ and prove that, similarly
to those on the projective plane, any $\mathcal{F}$ can be represented by a bi-homogeneous
polynomial affine $1$-form. In case $\mathcal{F}$ has isolated singularities,
we show that, for $ \delta=1 $, the singular scheme of $ \mcF $
does determine the foliation, with some exceptions that we describe, as is the case of
foliations in the projective plane. For $ \delta \neq 1 $, we prove that
the singular scheme of $\mathcal{F}$  does not determine the foliation. However
we prove that, in most cases, two foliations $\mathcal{F}$
and $\mathcal{F}'$ given by sections $s$ and $s'$ have the same singular scheme if and only
if $ s' = \Phi( s )$, for some global endomorphism $ \Phi $ of the tangent bundle of $S_\delta$.
\end{abstract}
\maketitle
\section{Introduction}\label{La1}
The study of complex planar polynomial differential systems goes back to the XIX century.
Articles by Autonne \cite{aut}, Darboux \cite{dar}, Painlev\'e \cite{pai} and Poincar\'e
\cite{poi1, poi2} can be considered as seminal references for this topic. Problems proposed
more than a century ago, as to obtain conditions for the existence of first integrals for the
above mentioned systems, are still pending for resolution. Considering holomorphic foliations
by curves with singularities (\textit{foliations} in the sequel) on the complex projective
plane have produced important advances in the knowledge of those systems \cite{ce-li,car,ca-ca,zam1,pere,l-n,es-kl,g-m-1,g-m-2,g-m-3,FGM}. Foliations can be defined
on another varieties extending the problems from the projective plane to those varieties
\cite{soa1, jou2,soa2, zam2, c-l, Correa}. Focusing on foliations on surfaces, Hirzebruch surfaces
$S_{\delta}$ with $\delta \neq 1$ (see Section \ref{La2} for our notation)
constitute jointly with the projective plane the
classical minimal rational surfaces, and the study of foliations on them is the first single step
after that on the projective plane. Our aim is to deepen the study of foliations
on Hirzebruch surfaces which have been treated within more general situations: as foliations
on ruled surfaces (in the profound monograph \cite{Gomez-Mont}) or as foliations
on toric varieties \cite{Correa}.

Let $ M $ be a compact connected complex manifold. Recall that a
foliation $\mathcal F$ on $ M $ 
may be defined by
non-identically zero holomorphic vector fields $ X_i $ defined on a covering $ \{ V_i \} $ of $ M $
such that in each overlapping set $ V_i \cap V_j $ we have
\begin{equation}\label{cocicle}
   X_i = \xi_{ij} X_j,
\end{equation}
where $ \xi_{ij} $ is a never vanishing holomorphic function. If
$ L^* $ denotes the holomorphic line bundle constructed with the cocycle $( \xi_{ij} )$,
and $\mcL^*$ its corresponding invertible sheaf,
then the $ X_i $'s give rise to a global section
$ s \in H^{0}(  M , \ctM \otimes \mcL^{*}) $ or to a global section in
$ H^{0}( M, \mathrm{Hom}_{ \mcO_M} (\mcL , \ctM) )$, where
$ \ctM $ is the tangent sheaf of $ M $ and $\mcL$ is the dual
of $\mcL^{*}$.
 Two global sections (in the corresponding spaces) define the same
foliation if and only if one is a non-zero scalar multiple of the other.

Following a somehow standard use (see  \cite{Brunella}, for instance) 
$ L^* $ will be called the \emph{cotangent bundle of} $\mathcal F$ and its dual 
$ L$, \emph{its tangent bundle}. Hence, the space $\mathrm{Fol}(\mcL, M) $ of foliations
$ \mcF $ with tangent bundle $ L $ (or tangent sheaf $ \mcL $)
 is $ \PP H^{0}( M, \mathrm{Hom}_{ \mcO_M} (\mcL , \ctM) )$.
Such an $ \mcF $ corresponds to a foliation with cotangent bundle $ L^*$ (or cotangent sheaf $ \mcL^* $)
by regarding it as the class $ [ s ] \in \PP H^{0}( M, \ctM \otimes \mcL^{*}) $ 
of a global section $ s \in H^{0}( M, \ctM \otimes \mcL^{*}) $.


Given a global section
$ s \in H^{0}( M, \mathrm{Hom}_{ \mcO_M} (\mcL , \ctM) )$,
the scheme $ Z = Z_s $ of those points 
$p\in M $ where the induced morphism
$\mcL_p\rightarrow \Theta_{M,p}$ becomes zero will be referred to as
the \emph{singular scheme} of $ s $: its sheaf of ideals $ I_Z \subset \mcO_M $
is the sheaf obtained by gluing the ideals $ (a_i, b_i) \subset \mcO( V_i ) $,
where $ a_i $ and $ b_i $ are the coefficients of the vector field $ X_i $ that defines
$ s $ on the open set $ V_i $, as described nearby \eqref{cocicle}. The singular scheme
of $ \mcF = [s] $ is the singular scheme of any section in $ [s] $.

We say that $ \mcF = [s] $ \emph{has isolated singularities} if
$ \text{dim } Z_s = 0 $.

In a series of papers \cite{Cam-Oli-1, Cam-Oli, Cam-Oli-2}, with the precedent of \cite{GM-2},
Campillo and the third author have proved that given a foliation
$ [s] \in \PP H^0( \PP^n, \mathrm{Hom}( \mcO_{\PP^n}( -d ), \Theta_{\PP^n} ) )
 = \mathrm{Fol}( \mcO_{\PP^n}( -d ), \PP^n ) $
 with isolated singularities and $ d > 1 $,
 $ [s] $ is the unique foliation in $\mathrm{Fol} ( \mcO_{\PP^n}( -d ), \PP^n ) $
 with singular scheme $ Z = Z_s $. We summarize this statement saying that a foliation with isolated
 singularities of degree $ d > 1 $ in a projective space of dimension $ n \geq 2 $ is uniquely
 determined by its singular scheme.

In this paper we study the extension of this result to foliations with isolated singularities 
on Hirzebruch surfaces $ S_{\delta} $, with $ \delta \geq 0 $.

Other results of this type,
dealing with foliations (or distributions)
of rank different from $1$ in projective spaces, are given in
\cite{A-C}, \cite{C-F-N-V} and in \cite{G-P}.

To state our results, we recall from Section \ref{La2}  below that every invertible sheaf
$ \mcL $ on $ \cs $ has the form $\mcO_{\cs} (-d_1,-d_2) $, for some $ d_1, d_2 \in \mathbb{Z}$. Hence,
foliations on $ \cs $ come from sections $ s \in H^0( \cs, \mathrm{Hom} (\mcO_{\cs} (-d_1,-d_2), \cts ) ) $.

We work within the toric structure of $ \cs $. This point of view gives us a way to represent
every section $ s $ by a bi-homogeneous polynomial affine $1$-form $ \Omega $ on
$(\mathbb{C}^2\setminus \{0\})\times (\mathbb{C}^2\setminus \{0\})$ --an affine $1$-form, for
short-- in essentially the same way as a projective $1$-form (say) in $ \PP^2 $ is representable by a
polynomial homogeneous $1$-form in
$(\mathbb{C}^3\setminus \{0\}) $ (see Proposition \ref{forms}).
This representation is one of our main tools. On the one hand, because it allows us to prove that if
$ s $ has isolated singularities, then $ d_1 \geq 0 $ and $ d_2 \geq 0 $, if $ \delta = 0 $ and
$ d_1 \geq -1 $ and $ d_2 \geq 0 $, if $ \delta \geq 1 $
(see Proposition \ref{isolated}, which is a refinement of \cite[Proposition 2.2]{Gomez-Mont}). On the
other hand, because the coefficients of $ \Omega $ generate the ideal of the singular scheme $ Z $
of $ s $ (see Remark \ref{ShfofIds}).

Global endomorphisms $ \Phi $ of $ T\cs $ play a central role. To start, Corollary \ref{manyfols} shows that,
for $ \delta \neq 1 $ and $ \Phi $ invertible, all foliations $ [ \Phi( s ) ] $ have the same singular
scheme as $ [s] $ does and most of them are different from $ [s] $;
therefore in this case, the singular scheme
does not determine the foliation on the contrary to what happens in the projective case.
However, we prove in Theorem \ref{mainT} that for $ d_2 \geq 1 $
and $d_1 \geq 1$ (in case $\delta=0$), and $d_1 \geq 2$ (in case $\delta \geq 2$),
these foliations $ [ \Phi( s ) ] $ are the unique ones that
share singular scheme with $ [s] $. We prove moreover that, in case $ \delta = 1 $, the foliation $ [s] $
is uniquely determined by its singular scheme if $ d_2 \geq 1 $ and $ d_1 \geq 0 $.

Some preliminaries on Hirzebruch surfaces that will be used along the paper are given in Section \ref{La2}.
In Section \ref{La3}
we give the aforementioned representation of the sections $ s \in H^{0}( S_{\delta}, \cts \otimes \mcL^{*}) $
in terms of polynomial bi-homogeneous affine $1$-forms, as well as one in terms of
vector fields. Theorem \ref{endtang} on the structure of the global endomorphisms of the tangent bundle
of Hirzebruch surfaces is the main content of Section \ref{La4} and we see it as one of our
leading results. Theorem \ref{mainT}, the main result of the paper, is proved in Section \ref{La5}.
\subsection{Notation} Throughout the paper, the structure sheaf $ \mcO_{\cs} $ of $ \cs $ will be denoted by
$ \mcO $. The sheaves of sections of the tangent $ T\cs $ and cotangent $ T^*\cs $ bundles will be denoted
respectively by $ \Theta_{\cs} $ and $ \Omega_{\cs}^1 $. For a line bundle $ L $
its associated invertible sheaf will be denoted by $ \mcO( L ) $. If it has the form
$ \mcL = \mcO(d_1, d_2) $, then the sheaves of sections of the bundles $ T\cs \otimes L $ and
$ T^*\cs \otimes L $ will be denoted respectively by $ \Theta_{\cs}(d_1, d_2) $ and $ \Omega_{\cs}^1(d_1, d_2) $.

\section{Preliminaries on Hirzebruch surfaces}\label{La2}
For an integer $\delta \geq 0 $, the
Hirzebruch surface $S_{\delta}$ is the ruled surface
\begin{equation}\label{ruled}
  \psi: S_{\delta}\rightarrow \mathbb{P}^1
\end{equation}
associated to
$ \mathbb{P}(\mcO_{\mathbb{P}^1}\oplus \mcO_{\mathbb{P}^1}(-\delta)) $, where
$ \mathbb{P}^1 $ is the complex projective line
\cite[Chapter V, Corollary 2.13]{Hartshorne}.

The surjective map $ \psi $
gives $S_{\delta}$ the structure of a $\mathbb{P}^1$-bundle over $\mathbb{P}^1$ and its
fibres constitute \emph{the ruling} of $S_{\delta}$.

 Let $F$ and $M$ be two generators of the
 divisor class group 
 ${\rm Cl}(S_{\delta})$ such that
 $F^2=0$, $M^2=\delta$ and $F\cdot M=1$. If $\delta>0$, let
 $M_0$ denote the class of the $(-\delta)$-curve of $S_{\delta}$,  that is,
 the unique irreducible curve of $S_{\delta}$ with negative self-intersection
 (if $\delta=0$ , take $ M $ for $ M_0 $). For simplicity, for each $E\in {\rm Cl}(S_{\delta})$,
 $E$ will also denote the image of $E$ in ${\rm Cl}(S_{\delta})\otimes \mathbb{Q}$. For
 every $ d_1, d_2\in \mathbb{Z} $, the invertible sheaf $\mcO(d_1, d_2)$ corresponds to the class
 $ d_1 F + d_2 M $.

The \emph{cone of curves} $NE(S_{\delta})$ of $S_{\delta}$ is the convex cone of
${\rm Cl}(S_{\delta})\otimes \mathbb{Q}$
generated by the images of the effective classes.
Its dual cone  $NE(S_{\delta})^\vee$ (with respect to the intersection form) is called the
\emph{nef cone} and is denoted by $P(S_{\delta})$. Specifically:
\[
 P(S_{\delta}):=NE(S_{\delta})^\vee=\{E\in {\rm Cl}(S_{\delta})\otimes \mathbb{Q}\mid E\cdot C\geq 0\;
 \mbox{for any effective divisor $C$ on $S_{\delta}$}\}.
\]
The \emph{ample cone}
Amp$(S_{\delta})$
of $S_{\delta}$ is the convex cone of
${\rm Cl}(S_{\delta})\otimes \mathbb{Q}$
whose elements are the ample classes. These classes are described in the following proposition
(whose proof can be deduced from \cite[Chapter V, Corollary 2.18]{Hartshorne}):
\begin{proposition}\label{ample}
A class $ d_1 F + d_2 M $ of ${\rm Cl}(S_{\delta})$ is ample if and only if $d_1, d_2 > 0 $.
\end{proposition}

Since Amp$(S_{\delta})$ is the topological interior of $P(S_{\delta})$ (see \cite{Kleiman}) it
follows from Proposition \ref{ample} that $P(S_{\delta})$ is the convex cone spanned by $F$ and $M$.
Moreover, the topological closure of $NE(S_{\delta})$ is equal to $P(S_{\delta})^\vee$ and, therefore,
it is the convex cone spanned by $F$ and $M_0$; since both generators are effective, one has that
$NE(S_{\delta})$ is closed and it is spanned by the classes $F$ and $M_0$.  From these facts, the following
result is clear.

\begin{proposition}\label{effective}
 A class $ d_1 F + d_2 M $ in ${\rm Cl}(S_{\delta})$ is effective if and only if
 $d_1+\delta d_2 \geq 0$ and $d_2 \geq 0$.
\end{proposition}
The Hirzebruch surface $S_{\delta}$  also has the structure of a toric variety,
that is, it can be regarded as the quotient of
$(\mathbb{C}^2\setminus \{0\})\times (\mathbb{C}^2\setminus \{0\})$ by an action of the algebraic torus
$(\mathbb{C}\setminus \{0\})\times (\mathbb{C}\setminus \{0\})$. Indeed, considering coordinates
$(X_0,X_1,Y_0,Y_1)$ in $(\mathbb{C}^2\setminus \{0\})\times (\mathbb{C}^2\setminus \{0\})$, the action is
given by
\[
(\lambda,\mu)\cdot (X_0,X_1,Y_0,Y_1):=(\lambda X_0,\lambda X_1,\mu Y_0,\lambda^{-\delta}\mu Y_1),
\]
for all $(\lambda,\mu)\in (\mathbb{C}\setminus \{0\})\times (\mathbb{C}\setminus \{0\})$ (see \cite{Correa},
where $ S_{\delta} $ appears as $ \mathbb{F}(0, \delta) $).
 Thus, we have a natural quotient map
 \begin{equation}\label{quotmap}
    \pi: (\mathbb{C}^2\setminus \{0\})\times (\mathbb{C}^2\setminus \{0\})\rightarrow S_{\delta}.
 \end{equation}
For integers $d_1$ and $d_2$,
a polynomial $H(X_0,X_1,Y_0,Y_1)\in \mathbb{C}[X_0,X_1,Y_0,Y_1]$ is said to be
\emph{bi-homogeneous of bi-degree} $(d_1,d_2)$ if
every monomial $X_0^\alpha X_1^\beta Y_0^\gamma Y_1^\mu$ appearing in $H$ with non-zero coefficient satisfies
that $\alpha+\beta-\delta\mu=d_1$ and $\gamma+\mu=d_2$. For any effective divisor $ d_1 F + d_2 M $ in $S_{\delta}$,
the non-zero global sections of $\mcO(d_1,d_2)$ correspond to bi-homogeneous polynomials of bi-degree
$(d_1,d_2)$.

Given a line bundle $ L $ on $S_{\delta}$, the Chern class
$ c( L ) = a f + b h \in H^2( S_{\delta}, \mathbb{Z} ) $ (considered below
\cite[Definition 1.1]{Gomez-Mont}) is expressed in bi-degree
form by setting $ e = \delta, f = F $ and $ f^{\prime} = M $ (so that $ [B_0]^* = M_0 $).
Then, it holds that
\begin{equation}\label{CoC}
 \left(
  \begin{matrix}
    d_1  \\
    d_2
  \end{matrix} \right) =
 \left(
  \begin{matrix}
    1 & -\delta/2  \\
    0 & 1
  \end{matrix} \right)
 \left(
  \begin{matrix}
      a  \\
      b
  \end{matrix} \right).
\end{equation}
For instance, the Chern class $ c( K_{\cs} ) = (2g-2) f -2 h = -2f -2h $ of the canonical
bundle $ K_{\cs} $ \cite[Lemma 1.3]{Gomez-Mont} corresponds to the
\textit{canonical sheaf} $ \mcK_{\cs} = \mcO(\delta-2,-2) $.

The surface $S_{\delta}$ is covered by the four affine open sets $U_{ij}$, $i,j\in \{0,1\}$, given by
\begin{equation}\label{affcov}
 U_{ij}:=\{\pi(X_0,X_1,Y_0,Y_1)\in S_{\delta}\mid X_i\neq 0 \mbox{ and } Y_j\neq 0\},
\end{equation}
where $ \pi $ is the quotient map \eqref{quotmap}.
Since $ \pi(X_0,X_1,Y_0,Y_1) = \pi(1,X_1/X_0,1,X_0^{\delta}Y_1/Y_0) $ in $ U_{00}$, the open set
$U_{00}$ is identified with $\mathbb{C}^2$ by means of the isomorphism:
\[
 \pi(X_0,X_1,Y_0,Y_1)\mapsto (x_{00},y_{00}),
\]
where $x_{00}:=X_1/X_0$ and $y_{00}:=X_0^{\delta}Y_1/Y_0$. Similarly $U_{10}$ is identified with
$\mathbb{C}^2$ by means of the isomorphism $\pi(X_0,X_1,Y_0,Y_1)\mapsto (x_{10},y_{10})$, where
$x_{10}:=X_0/X_1$ and $y_{10}:=X_1^{\delta}Y_1/Y_0$. The change of coordinates map in the overlap
of $U_{00}$ and $U_{10}$ is given by
\[
\varphi_{00}^{10}:U_{00}\cap U_{10}\subseteq U_{00}\rightarrow U_{00}\cap U_{10}\subseteq U_{10},\;\; (x_{00},y_{00})\mapsto (1/x_{00},x_{00}^\delta y_{00})=(x_{10}, y_{10}).
\]
If $C$ is the curve on $S_{\delta}$ defined by the zero
locus of a bi-homogeneous polynomial $H(X_0,X_1,Y_0,Y_1)$ then
the intersection $ C\cap U_{00} $ is the zero locus of the polynomial
in the affine coordinates $x_{00}$ and $y_{00}$ given by
\begin{equation}\label{restriction}
  \tilde{H}^{00}(x_{00},y_{00}):=H(1,x_{00},1,y_{00}).
\end{equation}
Analogously, for each $i,j\in \{0,1\}$, we can obtain affine coordinates
$(x_{ij},y_{ij})$ for every affine open set $U_{ij}$, change of coordinates maps
\begin{equation}\label{chofcoords}
  \varphi_{ij}^{i'j'}:U_{ij}\cap U_{i'j'}\subseteq U_{ij}\rightarrow U_{ij}\cap U_{i'j'}\subseteq U_{i'j'},
\end{equation}
and an equation $\tilde{H}^{ij}=0$ for the intersection of $C$ with $U_{ij}$.
\section{Representation of foliations on $ S_{\delta} $ by affine vector fields and $1$-forms}\label{La3}
Recall from Section \ref{La1} that a foliation $\mathcal{F}$ on the Hirzebruch surface $S_\delta$
is given by the class $[s]$ of a global section $ s \in H^{0}( S_{\delta}, \cts \otimes \mcL^{*}) $.
Let $\mathcal L=\mcO(-d_1,-d_2)$ be the tangent sheaf of $\mathcal{F}$. We follow \cite[\S 3.1]{Correa} to obtain the vector field representation of $ s $. To that end, define
\[
 \mcH : = \mcO(1,0)^{\oplus 2} \oplus \mcO(0,1)\oplus \mcO(-\delta,1),
\]
and  consider the \textit{Euler} exact sequence given by
\begin{equation}\label{euler}
0\rightarrow \mcO^{\oplus 2}\xrightarrow{j} \mcH \xrightarrow{d \pi} \cts\rightarrow 0.
\end{equation}
Taking tensor product with $\mcL^* = \mcO(d_1, d_2) $ in the sequence \eqref{euler}, we obtain the exact sequence
\begin{equation}\label{twiseuler}
0\rightarrow {\mcO(d_1, d_2)}^{\oplus 2}
  \xrightarrow{j\otimes 1} \mcH(d_1, d_2)
   \xrightarrow{d \pi\otimes 1} \cts(d_1, d_2)\rightarrow 0.
\end{equation}
The long exact sequence associated to \eqref{twiseuler} reads
\begin{equation}\label{LeCtan}
 \begin{split}
    0 & \rightarrow H^0(S_{\delta},\mcO(d_1, d_2))^{\oplus 2}
 \xrightarrow{{j\otimes 1}^0}
  H^0(S_{\delta}, \mcH(d_1, d_2) )
   \xrightarrow{{d \pi\otimes 1}^0} H^0(S_{\delta},\cts(d_1, d_1) ) \xrightarrow{\delta^0}  \\
      &      \xrightarrow{\delta^0}
      H^1(S_{\delta},\mcO(d_1, d_2))^{\oplus 2} \xrightarrow{{j\otimes 1}^1} H^1(S_{\delta}, \mcH(d_1, d_2) )
      \xrightarrow{{d \pi\otimes 1}^1} H^1(S_{\delta},\cts(d_1, d_2) ) \xrightarrow{\delta^1}
       \cdots ,
 \end{split}
\end{equation}
where
\begin{equation}\label{CohomTan}
  H^q(\cs, {\mcH}(d_1, d_2) ) = 
   H^q(\cs, \mcO(d_1 + 1, d_2) )^{\oplus 2} \oplus H^q(\cs, \mcO(d_1, d_2 + 1) )\oplus H^q(\cs, \mcO(d_1-\delta, d_2+1)),
\end{equation}
for $ q = 0,1,2 $ and
$ ({j\otimes 1}^0)(H_1,H_2) = (X_0H_1, X_1H_1, Y_0H_2, -\delta Y_1 H_1+Y_1H_2) $. The sequence \eqref{LeCtan} has the following interpretation:

Any section $ s $ \emph{in the image of} $ {d \pi\otimes 1}^0 $ is uniquely determined by a
vector field
\begin{equation}\label{fields}
 X = V_0\frac{\partial}{\partial X_0}+V_1\frac{\partial}{\partial X_1}
     + W_0\frac{\partial}{\partial Y_0}+W_1\frac{\partial}{\partial Y_1},
\end{equation}
where
$ V_0, V_1 \in H^0(S_{\delta},\mcO(d_1+1,d_2)) $,
$ W_0 \in H^0(S_{\delta},\mcO(d_1,d_2+1))$ and
$ W_1 \in H^0(S_{\delta},\mcO(d_1-\delta,d_2+1))$,
up to the addition of multiples of the \emph{radial} vector fields
$R_1:=X_0\frac{\partial}{\partial X_0}+X_1\frac{\partial}{\partial X_1}
 -\delta Y_1\frac{\partial}{\partial Y_1}$ and
 $R_2:=Y_0\frac{\partial}{\partial Y_0}+Y_1\frac{\partial}{\partial Y_1}$.
\begin{remark}\label{RepByVF}
We say for brevity that a section
$ s \in H^{0}( S_{\delta}, \cts \otimes \mcL^{*}) $
is \emph{representable by an affine vector field} $X$
if $ s $ lies in the image of the map $ {d \pi\otimes 1}^0 $ from \eqref{LeCtan}.
If this is the case, vector fields $ X_{ij} $ that define $ s $ in the covering \eqref{affcov} may be computed by
writing the product $ d \pi \cdot X $ in the coordinates $(x_{ij},y_{ij})$ described above.

Of course, \emph{every} section $ s $ is representable by an affine vector field if and only if the map
$ {d \pi\otimes 1}^0 $ is surjective, and this is the case \emph{if} (but not \emph{only if})
$ h^1(S_{\delta},{\mathcal L^*}) = 0 $ (see Remark \ref{tauene} below).
\end{remark}

Foliations on $ \cs $ may be also defined in terms of $1$-forms. Indeed, considering the covering
$ \{ V_i \} $ of $ \cs $ and vector fields $ X_i $ associated to a global section
$ s \in H^{0}( S_{\delta}, \cts \otimes \mcL^{*}) $ as in \eqref{cocicle}, the $1$-form associated to $ s $
is given by a collection of $1$-forms $ \Omega_i $ on $ V_i $ such that $ \Omega_i( X_ i ) = 0 $. These $1$-forms glue together into a global section (\emph{the annihilator of $ s $}) in
$ H^0( \cs, \Omega_{\cs}^1( d_1 + 2 - \delta, d_2 + 2) ) $.

Now we proceed with this construction. Let $\mcL=\mcO(-d_1,-d_2) $, with $d_1,d_2\in \mathbb{Z}$.

 First, we see from \cite[Section II, Exercise 5.16 (b)]{Hartshorne} applied to $ \mcF = \Theta_{\cs} $
 that the evaluation map
  $ b : \Theta_{\cs} \times \wedge^2 \Omega_{\cs}^1 \longrightarrow \Omega_{\cs}^1; $
  $  b( X, \omega) = \omega( X ) $,  induces an isomorphism
 $$ \Theta_{\cs}\otimes \mcK_{\cs} = \Theta_{\cs} \otimes \wedge^2 \Omega_{\cs}^1
 \xrightarrow{\tilde{b}} \Omega_{\cs}^1,
 $$
 which gives in turn an isomorphism 
 $ \Theta_{\cs} \cong
 \Omega_{\cs}^1 \otimes \mcK_{\cs}^* = \Omega_{\cs}^1(2-\delta, 2) $, hence
 $$ \Theta_{\cs}\otimes \mcL^* = \Theta_{\cs}( d_1, d_2 ) \cong \Omega_{\cs}^1( d_1 + 2 - \delta, d_2 + 2), $$
and we obtain that
\begin{equation}\label{CTanCCotan}
  H^q( \cs, \Theta_{\cs}( d_1, d_2) )\cong H^q( \cs, \Omega_{\cs}^1( d_1 + 2 - \delta, d_2 + 2) ), \;
  \text{for}\; q = 0, 1, 2.
\end{equation}
Now we seek for bi-homogeneous polynomial affine $1$-forms that represent the sections of \linebreak
$ H^0( \cs, \Omega_{\cs}^1( d_1 + 2 - \delta, d_2 + 2) ) $. To that end, first we dualize \eqref{euler}:
\begin{equation}\label{dual}
   0\rightarrow \Omega^1_{\cs} \xrightarrow{{d\pi}^*} {\mcH}^*
      \xrightarrow{j^*} \mcO^{\oplus 2} \rightarrow 0,
\end{equation}
then we twist \eqref{dual} by $\mcO(d_1-\delta+2,d_2+2)$:
\begin{equation*}
0\rightarrow \Omega^1_{\cs}(d_1-\delta+2,d_2+2) \xrightarrow{{d\pi}^*\otimes 1}
    {\mcH}^*(d_1-\delta+2,d_2+2) \xrightarrow{j^*\otimes 1}
  \mcO(d_1-\delta+2,d_2+2)^{\oplus 2} \rightarrow 0,
\end{equation*}
and we consider the long exact sequence associated to the exact sequence above, part of which reads as follows:
\begin{equation}\label{LeCcotan}
\begin{split}
  0 & \rightarrow H^0(\cs, \Omega^1_{\cs}(d_1-\delta+2,d_2+2) )\xrightarrow{{d\pi}^*\otimes 1^0}
      H^0(\cs, {\mcH}^*(d_1-\delta+2,d_2+2) ) \xrightarrow{j^*\otimes 1^0} \\
    & \xrightarrow{j^*\otimes 1^0}  H^0(\cs,\mcO(d_1-\delta+2,d_2+2))^{\oplus 2}
      \xrightarrow{\delta^0} H^1(\cs, \Omega^1_{\cs}(d_1-\delta+2,d_2+2) )\xrightarrow{{d\pi}^*\otimes 1^1} \cdots,
\end{split}
\end{equation}
where
\begin{equation}\label{CohomCotan}
  \begin{split}
 & H^q(\cs, {\mcH}^*(d_1-\delta+2,d_2+2) ) = \\
 &  H^q(\cs, \mcO(d_1-\delta+1,d_2+2) )^{\oplus 2} \oplus H^q(\cs, \mcO(d_1-\delta+2,d_2+1) )\oplus H^q(\cs, \mcO(d_1+2,d_2+1)),
\end{split}
\end{equation}
for $ q = 0,1,2 $ and
$$ (j^*\otimes 1^0)(A_0,A_1,B_0,B_1) = (X_0A_0 +X_1A_1-\delta Y_1B_1, Y_0B_0+Y_1B_1).$$


As a consequence, we deduce that given an invertible sheaf $\mcL=\mcO(-d_1,-d_2)$ on $ \cs $,
a foliation
$ \mcF = [s] \in \PP H^{0}( S_{\delta}, \cts \otimes \mcL^{*}) $ may not be representable
by a polynomial affine vector field, but it is always representable by some
bi-homogeneous differential $1$-form:
\begin{proposition}\label{forms}
Let $\mcL=\mcO(-d_1,-d_2)$, with $d_1,d_2\in \mathbb{Z}$. Then any foliation
$\mathcal F$ in $\mathrm{Fol}(\mcL,S_{\delta})$ is uniquely determined (up to multiplication by
a non-zero scalar) by a differential $1$-form
 \begin{equation}\label{1-form}
\Omega=A_0\;dX_0+A_1\; dX_1+B_0\; dY_0+ B_1\;dY_1,
 \end{equation}
 where
 $ A_0,A_1 \in H^0( S_{\delta}, \mcO(d_1-\delta+1,d_2+2) ) $,
 $ B_0 \in H^0( S_{\delta}, \mcO(d_1-\delta+2,d_2+1) ) $ and
 $ B_1 \in H^0( S_{\delta}, \mcO(d_1+2,d_2+1) ) $
 are bi-homogeneous polynomials (not all of them equal to $0$)
that satisfy the following two conditions:
\begin{equation}\label{condit1}
   \begin{gathered}
  \Omega( R_1 ) = X_0A_0 +X_1A_1-\delta Y_1B_1 = 0, \mbox{and}  \\
  \Omega( R_2 ) = Y_0B_0+Y_1B_1 = 0.
   \end{gathered}
\end{equation}
Moreover, if $i,j\in \{0,1\}$, a differential $1$-form defining $\mathcal F$ in the affine open set $U_{ij}$
is given by
$ \Omega_{ij} = \tilde{A}_{i'}^{ij}dx_{ij}+\tilde{B}_{j'}^{ij}dy_{ij}$, where $\{i'\}:=\{0,1\}\setminus \{i\}$,
$\{j'\}:=\{0,1\}\setminus \{j\}$ and the correspondence $ H \mapsto \tilde{H} $ is given by
equation \eqref{restriction} and its ilk just below it.
\end{proposition}
Let $\tau $ be the kernel of the Jacobian of $\psi$ in \eqref{ruled}. It is a sub-line bundle of $TS_{\delta}$
and induces an exact sequence
\begin{equation*}
0\rightarrow \tau \rightarrow  TS_{\delta}\xrightarrow{d\psi} N\rightarrow 0
\end{equation*}
where $ N $ is the normal bundle to the ruling (see Equation (1.2) in \cite{Gomez-Mont}).
We see from \cite[Lemma 1.4]{Gomez-Mont} and \eqref{CoC} that
$ \mcO( \tau ) = \mcO(-\delta,2) $ and $ \mcO( N ) = \mcO(2,0) $, so that
the sequence above corresponds to
\begin{equation}\label{exact}
0\rightarrow \mcO(-\delta,2) \rightarrow
 \Theta_{S_{\delta}}\xrightarrow{d\psi} \mcO(2,0) \rightarrow 0.
\end{equation}
The spaces of foliations with tangent bundles $ \tau $ and $ N $ will play a role in the results that follow (specially in Proposition \ref{isolated} below). For this reason, our next three remarks gather information about them.

\begin{remark}\label{tauene}
We study the representation by affine vector fields
 of sections in $ H^0(S_{\delta}, \Theta_{S_{\delta}}(\delta,-2) ) $ and in
 $ H^0(S_{\delta}, \Theta_{S_{\delta}}(-2, 0) ) $, in the context of \eqref{LeCtan}.

Claim:
 $ h^1(S_{\delta}, \mcO(\tau^*) ) = h^1(S_{\delta}, \mcO(\delta,-2) ) $ and
 $ h^1(S_{\delta}, \mcO(N^*) ) = h^1(S_{\delta}, \mcO(-2,0) ) $  are equal to $ 1 \neq 0 $.
Indeed, let $ D = \delta F -2 M  $ and recall that $ K_{\cs} - D = -2 F $, so that
$ h^q(S_{\delta}, \mcO(\delta,-2) ) = h^{2-q}(S_{\delta}, \mcO(-2,0) ) $, for $ q = 0,1,2 $,
by Serre duality. Moreover,
\[
 h^0(S_{\delta}, \mcO(\delta,-2) )=0,  \text{  because $ -2 < 0 $ }
\]
and $ h^{0}(S_{\delta}, \mcO(-2,0) ) = h^{0}(\PP^1, \mcO_{\PP^1}(-2) ) = 0 $. Hence, the Euler
characteristic $ \chi( \mcO(\delta,-2) ) $ is equal, on the one hand, to $ - h^1(S_{\delta}, \mcO(\delta,-2) )$
and on the other hand, to $ \frac{1}{2} D\cdot(D - K_{\cs}) + \chi( \mcO ) = \frac{1}{2}(-4) + 1 = -1 $, by the Riemann-Roch theorem.

Having the claim, it follows from \eqref{LeCtan} that
\textit{no section $ s \in H^0(S_{\delta}, \Theta_{S_{\delta}}(\delta,-2) ) $
is representable by an affine vector field}. Indeed, we see from \eqref{CohomTan} that
$$ h^0(S_{\delta}, \mcH(\delta,-2)) =
   h^0(S_{\delta}, \mcO(1+\delta, -2))+
   h^0(S_{\delta}, \mcO(\delta, -1))+
   h^0(S_{\delta}, \mcO(0, -1)) = 0, $$
because each summand is equal to $ 0 $
(by the argument in the displayed equation above).
The best we can say from this computations on the value of
$ h^0(S_{\delta}, \Theta_{S_{\delta}}(\delta,-2) ) $ is that it is $ \geq 1 $ (from \eqref{exact})
and that it is $ \leq 2 $ (from \eqref{LeCtan}). See Remark \ref{tau} below for the actual value.

For the case of sections $ s \in H^0(S_{\delta}, \Theta_{S_{\delta}}(-2, 0) ) $, we have already seen (a couple of lines above) that $ h^0(S_{\delta}, \mcO(-2, 0) ) = 0 $.
Similar computations to the ones above 
show that
$$ h^0(S_{\delta}, \mcH(-2,0)) =
   h^0(S_{\delta}, \mcO(-1, 0))+
   h^0(S_{\delta}, \mcO(-2, 1))+
   h^0(S_{\delta}, \mcO(-(\delta+2), 1)) = h^0(S_{\delta}, \mcO(-2, 1)), $$ which is
equal to $ 0 $ for $ \delta = 0, 1 $, and it is equal to $ \delta-1 $ for $ \delta \geq 2 $.
\textit{The conclusion is that no
such a section $ s $ is representable by an affine vector field for $ \delta = 0, 1 $, and that
$ h^0(S_{\delta}, \Theta_{S_{\delta}}(-2, 0) ) \geq \delta-1 $, for $ \delta \geq 2 $.}

We conclude this remark by saying that in  Remark \ref{ene} below we will show however that
$$
h^0(S_{\delta}, \Theta_{S_{\delta}}(-2, 0) )
 = \begin{cases}
     1, & \mbox{if } \delta = 0 \\
     0, & \mbox{if } \delta = 1 \\
     \delta-1, & \mbox{if } \delta \geq 2.
   \end{cases}
$$
In concern with Remark \ref{RepByVF}, this computation shows that
\textit{for $ \delta \geq 2 $,} in the corresponding exact sequence \eqref{LeCtan}, the map
$ d\pi \otimes 1^0 $ is not only injective but also surjective and hence
\textit{the map $ \delta^0 $ is the zero map.}
\end{remark}

\begin{remark}\label{tau}
{\it The unique foliation in} \textrm{Fol}$(\mcO(-\delta,2), S_{\delta})$
{\it is the one given by the ruling}
$S_{\delta}\rightarrow \mathbb{P}^1$. Indeed, by Proposition \ref{forms}, any foliation $\mathcal F$ in
\textrm{Fol}$(\mcO(-\delta,2), S_{\delta})$, is representable by an
affine differential $1$-form $ \Omega $ as in \eqref{1-form},
 where
 $ A_0,A_1 \in H^0( S_{\delta}, \mcO(1,0) ) $,
 $ B_0 \in H^0( S_{\delta}, \mcO(2,-1) ) $ and
 $ B_1 \in H^0( S_{\delta}, \mcO(\delta+2,-1) ) $
 are bi-homogeneous polynomials
that satisfy the conditions in \eqref{condit1}.
Since the last two sheaves have no non-zero global section, it follows that $B_0=B_1=0$ and
hence, $\mathcal F$ is defined by any non-zero scalar multiple of the differential form
$ \Omega_{\tau} = X_1\;dX_0-X_0\; dX_1 $, which corresponds to the ruling. Finally, we see
from the local expressions at the end of Lemma \ref{tau} that $ \mcF $ {\it has no singularities.}
\end{remark}

\begin{remark}\label{ene}
Now consider the case of foliations with tangent sheaf $\mcO(N)=\mcO(2,0)$. On the one hand, recall
from \cite[Proposition 2.4]{Gomez-Mont} that any $ \mcF $ in Fol$(\mcO(2,0), S_{\delta})$
\textit{with isolated singularities} is actually smooth (that is, it has no singularities at all).
On the other hand, Brunella in \cite{Brunella} (as quoted in \cite{L-P}) states that a rational
surface $ Z $ carries a smooth holomorphic foliation $\mathcal{G}$ if and only if $Z$ is a Hirzebruch
surface and $\mathcal{G}$ a rational fibration. With these facts in mind, it should be clear that
$ \mcF $ is a rational fibration only if $ \delta = 0 $. Now we prove it:


It follows from Proposition \ref{forms} that $\mathcal F$ in Fol$(\mcO(2,0), S_{\delta})$ is
representable by an affine differential $1$-form $ \Omega $ as in \eqref{1-form},
where
 $ A_0,A_1 \in H^0( S_{\delta}, \mcO(-(\delta+1),2) ) $,
 $ B_0 \in H^0( S_{\delta}, \mcO(-\delta,1) ) $ and
 $ B_1 \in H^0( S_{\delta}, \mcO(0,1) ) $
 are bi-homogeneous polynomials that satisfy \eqref{condit1}. We distinguish three cases:

- Case 1: $\delta=0 $. As in Remark \ref{tau}, one can prove
 that {\it the unique foliation $ \mcF $ in {\rm Fol}$(\mcO(2,0), S_{0})$ is the one defined by
 $ \Omega = Y_1 dY_0 - Y_0 dY_1 $}: the ruling of
 $ S_{0}=\mathbb{P}^1\times \mathbb{P}^1$ with respect to the projection onto the second factor.
 Hence, $ \mcF $ is smooth.

- Case 2: $\delta=1 $. The complex vector spaces $ H^0( S_{1}, \mcO(-2,2) ) $ and
 $ H^0( S_{1}, \mcO(-1,1) ) $ consist of
 the scalar multiples of $ Y_1^2 $ and $ Y_1 $, respectively, and
 $ H^0( S_{1}, \mcO(0,1) ) $ contains only linear forms in $ Y_0 $ and $ Y_1 $. Hence, the affine $1$-forms
 \eqref{1-form} have the following shape:
 $$ \Omega = a Y_1^2 dX_0 + b Y_1^2 dX_1 + a_1 Y_1 dY_0 + Y_1 ( c_0 X_0 + c_1 X_1 ) dY_1, $$
where $ a, b, a_1, c_0, c_1 \in \CC $. Thus it follows easily that {\it the unique
 $ \Omega $ that satisfies conditions \eqref{condit1} is $ \Omega = 0 $.}

- Case 3: $\delta \geq 2 $. {\it We claim that no
  foliation
  $\mathcal F$ in $\mathrm{Fol}(\mcO(2,0), S_{\delta})$ has
  isolated singularities.} Indeed, consider a
  $1$-form $ \Omega $ as in \eqref{1-form} that represents such a foliation.
  Let $D$ denote the divisor
  $-(\delta+1)F + 2M $. Then  $ D\cdot M_0 < 0 $ and
  $ (D-M_0)\cdot M_0 < 0 $, which means that $M_0$ is a double fixed component of the complete
  linear system $ |D| $ and, therefore, $ Y_1^2 $ divides both $ A_0 $ and $ A_1 $. Now, from the first equation in \eqref{condit1} one gets that $ Y_1 $ divides $ B_1 $ and,
  by the second equation therein, that $Y_1^2$ divides $B_0$: this is a contradiction unless
  $ B_0 = B_1 = 0 $, because $ B_0 \in H^0( S_{\delta}, \mcO(-\delta,1) )$ and the latter consists of the
  scalar multiples of $Y_1$. Thus, we conclude that
  $ \Omega = Y_1^2 ( A_0^{\prime} dX_0 + A_1^{\prime} dX_1 ) $ for some
  $ A_j^{\prime} \in H^0( S_{\delta}, \mcO(\delta-1,0) ) = H^0( \PP_1, \mcO_{\PP_1}(\delta-1) ) $,
  and, from the first equation in \eqref{condit1}, that
  \begin{equation}\label{FinalN}
    \Omega =
    \Omega_N =
    Y_1^2 A_{\delta-2}(X_0, X_1)  ( X_1 dX_0 - X_0 dX_1 ) = Y_1^2 A_{\delta-2}(X_0, X_1)
     \Omega_{\tau},
  \end{equation}
  for some $ A_{\delta-2} \in H^0( \PP_1, \mcO_{\PP_1}(\delta-2) ) $ (see Remark \ref{tau} above for
  $ \Omega_{\tau}$). We conclude from \eqref{FinalN} that ${\mathcal F}$ has no isolated singularities
  and, moreover, from \eqref{CTanCCotan}, that
  $$ h^0( \cs, \Theta_{\cs}(-2,0) ) =
    h^0( \cs, \Omega_{\cs}^1(-\delta, 2) ) =
    h^0( \PP_1, \mcO_{\PP_1}(\delta-2) ) = \delta - 1.$$
\end{remark}

Our next result is a refinement of \cite[Proposition 2.2]{Gomez-Mont}: It computes those tangent sheaves $ \mcL $ for which a foliation $ \mcF \in \mathrm{Fol}( \mcL, \cs ) $
 may have isolated singularities:
\begin{proposition}\label{isolated}
Let $d_1,d_2\in \mathbb{Z}$ and let $\mathcal L=\mcO(-d_1,-d_2)$ be an invertible sheaf on $S_{\delta}$
 such that there exists a foliation $ \mcF \in \mathrm{Fol}( \mcL, \cs ) $
 with isolated singularities. If $ \delta=0 $ (respectively, $ \delta\geq 1 $) then, either
 $\mcL\cong \mcO(\tau)$, or $\mcL\cong \mcO(N)$, or $d_1\geq 0$ and $d_2\geq 0$ (respectively, either
 $ \mcL\cong \mcO(\tau) $,  or $d_1\geq -1$ and $d_2\geq 0$).
\end{proposition}
\begin{proof}
Under the correspondence \eqref{CoC}, \cite[Proposition 2.2]{Gomez-Mont}
 states that if there exists a foliation $ \mcF \in \textrm{Fol}( \mcL, \cs ) $
 with isolated singularities, then
 either $ \mcL\cong \mcO(\tau) $, or $\mcL\cong \mcO(N)$, or $ (2+d_1) F + d_2 M $ belongs to the
 closure of Amp$(\cs)$ (which coincides with the nef cone $P(\cs)$). By Proposition \ref{ample}, the
 last condition is equivalent to the system of inequalities $d_1\geq -2, d_2\geq 0$. In the case $\delta=0$,
 the double ruling of $S_0$ shows that the mentioned system is equivalent to $d_1\geq 0$ and $d_2\geq 0$ and the statement for $ \delta = 0 $ has been proved.

Now assume that $\delta\geq 1$. First, $ \mcL $ cannot be isomorphic to $ \mcO(N) $ by Remark \ref{ene}.
Finally, assume that $d_1=-2$ and $d_2\geq 0$. We will show that every
$ \mcF \in \textrm{Fol}( \mcO(-2, -d_2), \cs ) $ has no isolated singularities. Indeed,
 by Proposition \ref{forms}, $\mathcal F$ is representable by an affine differential $1$-form
 $$
  \Omega = A_0\;dX_0+A_1\; dX_1+B_0\; dY_0+ B_1\;dY_1,
 $$
 where $ A_0, A_1 \in H^0( \cs, \mcO(-(1+\delta),d_2+2) ), B_0 \in H^0( \cs, \mcO(-\delta,d_2+1) ) $
 and
 $ B_1 \in H^0( \cs, \mcO(0,d_2+1) ) $ satisfy the conditions in \eqref{condit1}. If $D$ denotes the divisor $ -(1+\delta)F+(d_2+2)M $, it holds that $ D\cdot M_0=-(1+\delta) < 0 $
and $ (D-M_0)\cdot M_0 = -1 < 0 $. Therefore the complete linear system $|D|$ has $M_0$ as double fixed
component. This shows that $Y_1^2$ divides $A_0$ and $A_1$. It follows from \eqref{condit1} that $Y_1$
divides $B_1$ and $B_0$ as well, and hence that $ Y_1 $ is a factor of all the coefficients of $ \Omega $, which shows that $ \mcF $ has no isolated singularities. This finishes the proof.
\end{proof}

\begin{proposition}\label{lemma3}
 Let $\mcL=\mcO(-d_1,-d_2)$ be an invertible sheaf on $S_{\delta}$ such that
 $d_2\geq 0$ and, either $\delta=0$ and $d_1\geq 0$, or $\delta\geq 1$ and $d_1\geq -1$.
 Then $h^1(S_{\delta},\mcL^*)=0$.
\end{proposition}

\begin{proof}
 If $d_1,d_2\geq 0$ the result follows from \cite[Proposition 2.3]{Laface}. So let us assume
 $ \delta\geq 1 $, $d_1= -1$ and $d_2\geq 0$. Since $(-F+d_2M)\cdot M_0=-1$, it holds that $M_0$
 is a fixed component of the complete linear system $|-F+d_2M|$ and therefore
$$ h^0(S_{\delta},\mcO(-1,d_2))=h^0(S_{\delta},\mcO(\delta-1,d_2-1))
  =\frac{\delta}{2}d_2(d_2+1)=\chi(\mcO(-1,d_2)), $$
where the second equality comes again from \cite[Proposition 2.3]{Laface} and the third
from the Riemann-Roch theorem. The result follows from $ h^2(S_{\delta},\mcO(-1,d_2)) = h^0(S_{\delta},\mcO(\delta-1,-(d_2+2))) = 0 $
(by Serre duality).
\end{proof}
 \begin{remark}\label{ShfofIds}
Under the conditions of Proposition \ref{lemma3}, every foliation in Fol$(\mcL,S_{\delta})$
 has two equivalent descriptions: through affine vector fields (\ref{fields}) --in view of Remark \ref{RepByVF}-- and through some affine differential
 $1$-form \eqref{1-form}. Moreover, Proposition \ref{isolated} shows that this double
 description includes all foliations on $S_{\delta}$ with isolated singularities, except the one with tangent
 bundle $ \tau $ and, in the case $\delta=0$, also the one with tangent bundle $ N $ (see Remark \ref{tauene}
 above). However, the foliations associated to these exceptional cases are actually smooth (by Remarks \ref{tau} and \ref{tauene}, respectively).

Assume now that a section $ s \in H^{0}( S_{\delta}, \cts \otimes \mcL^{*}) $ is representable by an affine vector
field $ X $ as in \eqref{fields}, then the affine $1$-form $\Omega$  in \eqref{1-form} that corresponds
to $ s $ is given by
\begin{equation}\label{BigOmega}
\begin{aligned}
\Omega  & =
 \begin{vmatrix}
  dX_0 & dX_1 & dY_0 & dY_1\\
  X_0 & X_1 & 0 & -\delta Y_1\\
  0 & 0 & Y_0 & Y_1\\
  V_0 & V_1 & W_0 & W_1
 \end{vmatrix} \\
        & =
  \begin{vmatrix}
  X_1 & 0 & -\delta Y_1\\
  0 & Y_0 & Y_1\\
  V_1 & W_0 & W_1
  \end{vmatrix} dX_0 -
   \begin{vmatrix}
    X_0 & 0 & -\delta Y_1\\
      0 & Y_0 & Y_1\\
    V_0 & W_0 & W_1
    \end{vmatrix} dX_1 +
      \begin{vmatrix}
        X_0 & X_1 \\
        V_0 & V_1
      \end{vmatrix} (-Y_1 dY_0 + Y_0 dY_1)\\
        & =  A_0\;dX_0+A_1\; dX_1+B_0\; dY_0+ B_1\;dY_1.
    \end{aligned}
\end{equation}
 This follows because $ \Omega( R_1 ) = \Omega( R_2 ) = \Omega( X ) = 0 $.
 A further conclusion is that the sheaf of ideals $ I_Z $ of the singular
 scheme $ Z = Z_s $ of the section $ s $ is the ideal $ I_Z = (A_0, A_1, B_0, B_1) \subset \mcO $
 generated by the coefficients of $ \Omega $: this can be deduced from the local expressions for
 $ \Omega $ at the end of Proposition \ref{forms} together with the just proven fact that
 $ \Omega( X ) = 0 $.
 \end{remark}
\section{Global endomorphisms of $T\cs$}  \label{La4}
In this section we compute the space of global endomorphisms of the tangent bundle of a Hirzebruch surface. This computation will be essential to establish the main result of this paper.
\begin{theorem}\label{endtang}
 Consider the Hirzebruch Surface  $ S_{\delta} $, with $ \delta \geq 0 $.
 The space of global endomorfisms of its tangent bundle $T\cs$ has dimension
 $$
 h^0(S_{\delta},\mathrm{Hom}_{ \mcO} (\cts, \cts))=
    \begin{cases}
              2 & \mbox{if } \delta = 0\\
       \delta & \mbox{if } \delta \geq 1.
    \end{cases}
 $$
 Moreover, every such global endomorphism $\Phi$ is uniquely determined by a matrix $ A $ which is equal to
 \begin{equation*}
 \begin{array}{cccl}
   A(a, d) &= & a \cdot \mathbf{1}_{2\times 2} \oplus d \cdot \mathbf{1}_{2\times 2} &\text{ if } \delta = 0,  \\ \\
   A(a) &= & a \cdot \mathbf{1}_{4\times 4} &\text{ if } \delta = 1, \text{ and }\\ \\
   A(a,C) &= & \left( \begin{matrix} 
       a & 0 & 0 & 0 \\
       0 & a & 0 & 0 \\
       X_{1}Y_{1}C(X) & -X_{0}Y_{1}C(X)  & a & 0 \\
       0 & 0 & 0 & a
    \end{matrix} \right) &\text{ if } \delta \geq 2 ,
 \end{array}
\end{equation*}
 where $ a, d \in \mathbb{C} $, $ \mathbf{1}_{n\times n} $
 denotes the $ n\times n $ identity matrix ($n \in \mathbb{N}$) and $ C(X) = C(X_{0}, X_{1}) $ is a homogeneous polynomial of degree $ \delta - 2 $.

Moreover the following properties are satisfied:
 \begin{itemize}
\item[(a)] $\Phi$ is invertible if and only if $a d\neq 0$ if $ \delta = 0 $ and $a\neq 0$ if $ \delta \geq 1 $.
\item[(b)] For any invertible sheaf $ \mcL $ on $ \cs $, if \,
 $ \Phi \in \text{H}^0(S_{\delta},\mathrm{Hom}_{ \mcO} (\cts, \cts)) $ and\\
 $ s \in \text{H}^{0}( S_{\delta}, \mathrm{Hom}_{ \mcO} (\mcL , \cts) )$, then
 $ \Phi \circ s = \Phi( s )
 \in \text{H}^{0}( S_{\delta}, \mathrm{Hom}_{ \mcO} (\mcL , \cts) )$.

 In particular, if $ s $ is representable by the affine vector field
 \[
 X = V_0\frac{\partial}{\partial X_0}+V_1\frac{\partial}{\partial X_1}
    + W_0\frac{\partial}{\partial Y_0}+W_1\frac{\partial}{\partial Y_1},
 \]
 then $ \Phi(s) $ is representable by the affine vector field
 \[
 X':=V'_0\frac{\partial}{\partial X_0}+V'_1\frac{\partial}{\partial X_1}+W'_0\frac{\partial}{\partial Y_0}+W'_1\frac{\partial}{\partial Y_1},
 \]
 where $(V_0',V_1',W_0',W_1')^t = A \cdot (V_0,V_1,W_0,W_1)^t$ and the matrix $ A $ represents $ \Phi $.
 \footnote{The superscript $t$ denotes the transpose of the vector and the dot $ \cdot $, matrix multiplication.}
 \item[(c)] Under the hypothesis of representability in (b), let
 $ \Omega = A_0\;dX_0+A_1\; dX_1+B_0\; dY_0+ B_1\;dY_1 $ and
 $ \Omega' = A_0'\;dX_0+A_1'\; dX_1+B_0'\; dY_0+ B_1'\;dY_1 $ be the affine $1$-forms that represent the
 sections $ s $ and $ \Phi( s ) $ respectively, through \eqref{BigOmega} and let $ A $ represent $ \Phi $. Then
 $$( A_0',A_1',B_0', B_1' ) = ( A_0,A_1,B_0, B_1 )\cdot A. $$
 \end{itemize}
 The hypothesis of representability in (b) and (c) holds, in particular, for any  section $ s $ where the corresponding invertible sheaf $\mathcal L = \mcO(-d_1,-d_2)$ satisfies the conditions of Proposition \ref{lemma3}.

 \end{theorem}
\begin{proof}
We start with the computation of the matrix $ A $ associated to a global endomorphism $ \Phi $.
Consider the open covering
$\{U_{ij}\}_{0\leq i,j\leq 1}$ of $S_{\delta}$ from \eqref{affcov}, where
$ TS_{\delta} \vert_{U_{ij}} \simeq U_{ij} \times \mathbb{C}^{2} $ for all $i,j\in \{0,1\}$.
A global endomorphism $ \Phi $ of $ TS_{\delta} $ is given by a collection of $2\times 2$ matrices
$\{M_{ij}\simeq \Phi \vert_{U_{ij}} \}_{0\leq i,j\leq 1}$ such that
the entries of each matrix $M_{ij}$ are regular functions on $U_{ij}$ and, for every point $p$
belonging to an overlap $U_{ij}\cap U_{i'j'}$, we have
 \begin{equation}\label{transitiontangent}
   M_{i'j'}\!\vert_{p}=J_{ij}^{i'j'}\!\vert_{p}\cdot M_{ij}\!\vert_{p}\cdot (J_{ij}^{i'j'}\!\vert_{p})^{-1},
 \end{equation}
where $J_{ij}^{i'j'}$ denotes the Jacobian matrix of the change of coordinates map
$\varphi_{ij}^{i'j'}$  from \eqref{chofcoords} and
 $ B\vert_p $ denotes the matrix obtained by evaluating the entries of $ B $ at $ p $. Write
\begin{equation}\label{MU}
  M_{00} = \begin{pmatrix}
          a & b \\
          c & d
        \end{pmatrix},
\end{equation}
where  $a,b,c$ and $d$ are regular functions on $U_{00}$ (that is, they are given by polynomials in
$\mathbb{C}[x_{00},y_{00}]$). Considering an arbitrary point
$p=(x_{00},y_{00})=(x_{01},y_{01})\in U_{00}\cap U_{01}$
(given in coordinates in both open subsets), equation \eqref{transitiontangent} becomes
\begin{equation}\label{01}
M_{01}\vert_p = \begin{pmatrix}
             1 & 0 \\
             0 & -y_{01}^2
           \end{pmatrix}
        \begin{pmatrix}
          a(x_{01},y_{01}^{-1}) & b(x_{01},y_{01}^{-1}) \\
          c(x_{01},y_{01}^{-1}) & d(x_{01},y_{01}^{-1})
        \end{pmatrix}
        \begin{pmatrix}
             1 & 0 \\
             0 & -y_{01}^{-2}
           \end{pmatrix}$$ $$ =
        \begin{pmatrix}
               a(x_{01},y_{01}^{-1}) & -b(x_{01},y_{01}^{-1})\cdot y_{01}^{-2} \\
          - c(x_{01},y_{01}^{-1})\cdot y_{01}^2 & d(x_{01},y_{01}^{-1})
        \end{pmatrix}.
\end{equation}
Since the entries of $M_{01}$ must be regular functions on $U_{01}$ (polynomials in $\mathbb{C}[x_{01},y_{01}]$)
and $p\in U_{00}\cap U_{01}$ is arbitrary, we have that, necessarily,
$b=0, a(x_{00},y_{00})=a(x_{00}), d(x_{00},y_{00})=d(x_{00}) $ (that is, $a$ and $d$ depend only on $x_{00}$)
and the degree of $c$ in $ y_{00} $ is $ \leq 2 $.

Now we plug these conditions into \eqref{MU}
and compute \eqref{transitiontangent} with $(i,j)=(0,0)$, $(i',j')=(1,1)$ and $p=(x_{00},y_{00})=(x_{11},y_{11})$
being an arbitrary point in $U_{00}\cap U_{11}$. We obtain that
\begin{equation}\label{11}
M_{11}\vert_p =         \begin{pmatrix}
          a(x_{11}^{-1}) & 0 \\
          \delta y_{11} x_{11}^{-1}[a(x_{11}^{-1}) - d(x_{11}^{-1}) ] + y_{11}^2 x_{11}^{(\delta - 2)} c(x_{11}^{-1}, y_{11}^{-1} x_{11}^{\delta}) & d(x_{11}^{-1})
        \end{pmatrix}.
\end{equation}
Hence, reasoning as above, since $ M_{11} $ is defined by regular entries,
this is the case only if the functions $ a $ and $ d $ are constant and
these constants must be equal if $ \delta \geq 1 $ (this follows from the
lower-left entry of the matrix above).
 Moreover, since the polynomial $c$ expressed in coordinates $u$ and $v$ must have the shape
 $ c(u,v) = c_0(u) + c_1(u) v + c_2(u) v^2 $ for some univariate polynomials $ c_i(u) $,
 we see that the term $y_{11}^2 x_{11}^{-(\delta + 2)} c(x_{11}^{-1}, y_{11}^{-1} x_{11}^{\delta})$ comes
 from a regular function in $U_{11}$ if and only if $ c = 0 $ (respectively, $ c_0 = c_1 = 0 $ and
 $ c_2(u) $ has degree $ \leq \delta - 2 $) if $ \delta \in \{0, 1\}$ (respectively, if $ \delta \geq 2 $).

We have shown, so far, that the restriction of a global endomorphism $\Phi$ to the affine subset $U_{00}$
must be given by a matrix of the type
\begin{equation}\label{MUnew}
  M_{00} =
        \begin{cases}
              \begin{pmatrix}
          a & 0 \\
          0 & d
        \end{pmatrix}, & \text{if $  \delta = 0 $,} \\
       \begin{pmatrix}
          a & 0 \\
          0 & a
        \end{pmatrix}, & \text{if $  \delta = 1 $,  and}\\
        \begin{pmatrix}
          a & 0 \\
          c(x_{00}) y_{00}^2 & a
        \end{pmatrix}, & \text{if $  \delta \geq 2 $,}
    \end{cases}
\end{equation}
where $ a, d \in \mathbb{C} $ and $ c $ is a polynomial in one variable of degree $ \leq \delta-2 $.
Using \eqref{01} and \eqref{11} we deduce that, for $\delta\in \{0,1\}$, the matrices $M_{01}$ and
$ M_{11} $ coincide with $ M_{00} $ and that, for $\delta\geq 2$ we have:
\[
M_{01} =
        \begin{pmatrix}
          a & 0 \\
          -c(x_{01}) y_{01}^2 & a
        \end{pmatrix}\;\;\mbox{ and }\;\; M_{11}=\begin{pmatrix}
          a & 0 \\
          c(x_{11}^{-1})x_{11}^{\delta-2} y_{11}^2 & a
        \end{pmatrix}.
\]
We deduce similarly that $ M_{10} $ also coincides with $ M_{00} $ for $\delta\in \{0,1\}$
and that for $\delta\geq 2$, we have:
\[
 M_{10} =
        \begin{pmatrix}
          a & 0 \\
          -c(x_{10}^{-1}) x_{10}^{\delta-2} y_{10}^2 & a
        \end{pmatrix}.
\]
Noticing that every collection of four matrices as before (that is,
with $a,d$ and $c$ satisfying the given conditions) also satisfies the remaining conditions
from \eqref{transitiontangent} we conclude the part of the statement
concerning the dimension of the space of global endomorphisms of $T_{S_{\delta}}$ and Item (a).

The matrix representation $ A $ of $ \Phi $ also follows from these computations:
according to the different values of $ \delta $, we obtain the entries of the matrix $ A $ (say)
from equation \eqref{MUnew}: $ C( X ) $ is the homogeneous form associated to the polynomial
$ c(x_{00}) $. Then one verifies that, according to the different values of
$ \delta $, the restrictions of $ A $ to the open sets $ U_{00}, U_{01}, U_{11} $ and $ U_{10} $
coincide, respectively, with the matrices $ M_{00}, M_{01}, M_{11} $ and $ M_{10} $ described above.

Now we prove Item (b). Let $\Phi$ be  as before and consider
a foliation $\mathcal F = [s] $ in Fol$(\mcL,S_{\delta})$
such that $ s $ is representable by an affine
vector field $X$ (as in the statement).
The restriction $ s\vert_{U_{00}} $ of the section $ s $ to this open set is the vector field $ (d\pi \cdot X)\mid_{U_{00}} $,
 where $ \pi $ comes from \eqref{quotmap}, and it is given by
\[
  s\vert_{U_{00}}=\left(-x_{00}\tilde{V}_0^{00}+\tilde{V}_1^{00}\right)\frac{\partial}{\partial x_{00}}+\left(\delta y_{00} \tilde{V}_0^{00}-y_{00}\tilde{W}_0^{00}+\tilde{W}_1^{00}\right)\frac{\partial}{\partial y_{00}}.
\]
Then, the restriction of $\Phi( s )$ to $U_{00}$  can be computed by using the matrix $M_{00}$ as follows:
\[
\Phi( s )\vert_{U_{00}}
 = \begin{cases}
  a\left(-x_{00}\tilde{V}_0^{00}+\tilde{V}_1^{00}\right)\frac{\partial}{\partial x_{00}}+ d\left(\delta y_{00} \tilde{V}_0^{00}-y_{00}\tilde{W}_0^{00}+\tilde{W}_1^{00}\right)\frac{\partial}{\partial y_{00}}
              , & \text{if $  \delta = 0 $,} \\ \\
  a\left(-x_{00}\tilde{V}_0^{00}+\tilde{V}_1^{00}\right)\frac{\partial}{\partial x_{00}}+
  a\left(\delta y_{00} \tilde{V}_0^{00}-y_{00}\tilde{W}_0^{00}+\tilde{W}_1^{00}\right)\frac{\partial}{\partial y_{00}},
                & \text{if $  \delta = 1 $,  and}\\ \\
        a\left(-x_{00}\tilde{V}_0^{00}+\tilde{V}_1^{00}\right)\frac{\partial}{\partial x_{00}}
        +\\\left[ c(x_{00})y_{00}^2\left(-x_{00}\tilde{V}_0^{00}+\tilde{V}_1^{00}\right)
        + a\left(\delta y_{00} \tilde{V}_0^{00}-y_{00}\tilde{W}_0^{00}+\tilde{W}_1^{00}\right)\right]\frac{\partial}{\partial y_{00}},
                & \text{if $  \delta \geq 2 $.}
    \end{cases}
\]
This concludes the proof of (b) because the vector field $ (d\pi \cdot X')\vert_{U_{00}} $ coincides, in each case, with the above ones, and the same happens when considering
the remaining open sets $U_{10}$, $U_{01}$ and $U_{11}$.

The proof of (c) is a straightforward computation (whose details we omit). Notice that $ \Omega' $ is obtained from
\eqref{BigOmega} by replacing the last row $ (V_0,V_1,W_0,W_1) $ in the determinant therein by the row $ (V_0',V_1',W_0',W_1') $
where $ (V_0',V_1',W_0',W_1')^t = A \cdot (V_0,V_1,W_0,W_1)^t $.
\end{proof}
\begin{corollary}\label{manyfols}
 Let $\mathcal L=\mcO(-d_{1},-d_{2})$ be an invertible sheaf on $S_{\delta}$ such that every section
 $ s \in H^{0}( S_{\delta}, \cts \otimes \mcL^{*}) $ is representable by an affine vector field $ X $
 as in \eqref{fields}. Fix a section $ s $ and let $ Z $ be its singular scheme with sheaf of ideals
 $ I_Z $. Let $0\neq \Phi \in H^0(S_{\delta},\mathrm{Hom}_{ \mcO} (\cts, \cts) $ be a global
 endomorphism and let $ s^{\prime} = \Phi( s ) $ have singular scheme $ Z^{\prime} $. Then
 \begin{itemize}
 \item[(1)] $ Z \subseteq Z^{\prime} $ and $ Z = Z^{\prime} $ if $ \Phi $ is invertible.
 \item[(2)] Let $ A $ be the matrix associated to $ \Phi $ by Theorem \ref{endtang}. If $ \delta = 1 $ then
    $ \Phi $ is invertible and $ \mcF = [s] = [s^{\prime}] $. If $\delta \neq 1$ and $ \Phi$ is invertible, then the condition $ \mcF = [s] = [s^{\prime}] $
    is equivalent to the condition $ a = d $, for $ \delta = 0 $, and it is equivalent to the conditions
    $ C = 0 $ or $ C\neq 0 $ and $ X_0 V_1 - X_1 V_0 = 0 $, for $ \delta \geq 2 $.
  \item[(3)] If $ \Phi $ is not invertible and $ s^{\prime} \neq 0 $, then $ Z^{\prime} $ is one-dimensional.
\end{itemize}
\end{corollary}
\begin{proof} Let
 $ \Omega = A_0\;dX_0+A_1\; dX_1+B_0\; dY_0+ B_1\;dY_1 $ and
 $ \Omega' = A_0'\;dX_0+A_1'\; dX_1+B_0'\; dY_0+ B_1'\;dY_1 $ be the affine $1$-forms that represent the
 sections $ s $ and $ s' $, respectively (see \eqref{BigOmega}).

For every $ \delta \geq 0 $, Theorem \ref{endtang} Item (c) states, in matrix notation, that
 $$ ( A_0',A_1',B_0', B_1' ) = ( A_0,A_1,B_0, B_1 )\cdot A. $$
 By Remark \ref{ShfofIds}, this equality implies that $ I_{Z'} \subseteq  I_{Z} $, which shows that
 $ Z \subseteq Z' $ and  the first part in (1) is proved.

Now we divide the proof of
the second part in (1) and of the remaining statements into three cases:

- Case 1: $ \delta = 0 $. By \eqref{BigOmega}:
  $$ \Omega = ( Y_0 W_1 - Y_1 W_0 ) (X_1 dX_0 - X_0 dX_1) + ( X_0 V_1 - X_1 V_0 ) (-Y_1 dY_0 + Y_0 dY_1), $$
  and $ A = A(a,d) $, by Theorem \ref{endtang}; then it follows from  Item (c) therein that
 $$
 \Omega' = a\cdot ( Y_0 W_1 - Y_1 W_0 ) (X_1 dX_0 - X_0 dX_1) + d\cdot ( X_0 V_1 - X_1 V_0 ) (-Y_1 dY_0 + Y_0 dY_1).
 $$
Since $ ad \neq 0 $, it follows from Remark \ref{ShfofIds} that $ I_Z = I_{Z'} $ (and hence, that
 $ Z = Z' $). This finishes the proof of (1) in this Case 1. With respect to (2), it is obvious that $ [s] = [s^{\prime}] $ if and only if $ a = d $. Finally, if (say) $ a = 0 $ then
 $ \Omega' = d\cdot( X_0 V_1 - X_1 V_0 ) (-Y_1 dY_0 + Y_0 dY_1) \neq 0 $ and we see from the final statement of Proposition \ref{forms} that the restriction of $ \Omega' $ to the open set $ U_{00} $ is given by
 $$
 \Omega'_{00} = d\cdot \tilde{B}_{1}^{00}\left( x_{00}, y_{00} \right)\, dy_{00}
  = d\cdot y_{00} \left( \tilde{V}_{1}^{00} - x_{00} \tilde{V}_{0}^{00} \right) \, dy_{00},
 $$
 so that $ \{ (x_{00}, 0) \} \subset Z' \cap U_{00} $ and (3) follows.

- Case 2:  $ \delta = 1 $.  This case  follows at once from the fact (Theorem \ref{endtang})
 that $ A = A(a) = a \cdot \mathbf{1}_{4\times 4} $.

- Case 3: $ \delta \geq 2 $. Recalling the expression for $ \Omega $ from \eqref{BigOmega} and the fact that $ A = A(a,C) $,
 it follows from  Item (c) in Theorem \ref{endtang} that
 \begin{equation}\label{deltageq2}
  \begin{aligned}
    \Omega' & = a\cdot\Omega - Y_1 C B_0 \cdot (X_1 dX_0 - X_0 dX_1) \\
            & = a\cdot\Omega + (X_0 V_1 - X_1 V_0) Y_1^2 C \cdot (X_1 dX_0 - X_0 dX_1),
  \end{aligned}
 \end{equation}
 so that $ A_0' = a\cdot A_0 - X_1 Y_1 C B_0 $,
         $ A_1' = a\cdot A_1 + X_0 Y_1 C B_0 $ and $ B_j' = a\cdot B_j $, for $ j = 0, 1 $.

 If $ a \neq 0 $, then $ A_0 - \frac{1}{a} A_0' \in ( B_0 ) = ( B_0' ) \subset I_{Z'} $,
 so that $ A_0 \in I_{Z'} $. Similar arguments show that $ A_1 \in I_{Z'} $ and then
 $ I_{Z} \subset I_{Z'} $ which, together with the above, shows that $ Z = Z' $ and the proof of (1) is complete.

The second equality in \eqref{deltageq2} proves (2).

Finally, (3) follows from the observation that the $1$-form in the second equality in \eqref{deltageq2}
coincides with $ a\cdot\Omega + (X_0 V_1 - X_1 V_0) \cdot \Omega_N $, where $ \Omega_N $ comes from
\eqref{FinalN} within Remark \ref{ene}. There, it was shown that $\Omega_N $ does not have isolated
singularities and hence $ \Omega' $ either does not, whenever $ a = 0 $.
\end{proof}

\section{Foliations with isolated singularities on $S_{\delta}$ that share singular scheme} \label{La5}

If $ \mcF = [s] \in \PP H^{0}( S_{\delta}, \cts(d_1,d_2) ) $ is a foliation with isolated singularities on $ \cs $  then for every invertible endomorphism $ \Phi $ of $ T\cs $, all foliations $ [\Phi(s)] $ share singular scheme with $ [s] $, by Corollary \ref{manyfols}. Our main result, Theorem \ref{mainT}, states that these are the only ones, whenever $ d_2 \geq 1 $ and
 $ d_1 \geq 1 $ for $ \delta = 0 $;
 $ d_1 \geq 2 $ for $ \delta \geq 2 $ and
 $ d_1 \geq 0 $ for $ \delta = 1 $. In this last case, we see that $ [s] $ is uniquely determined by its
 singular scheme. Notice that this result holds for all foliations with ample cotangent bundle, with the exception of the cases $\delta \geq 2$ and $d_1=1$. We devote this final section to the proof of Theorem \ref{mainT}.


\begin{lemma}\label{Lemma 3}
Consider the families of invertible sheaves on $ S_{\delta} $ given by
 $ \mcL = \mcO( -d_1, -d_2 ) $ and $ \mcE = \mcL \otimes \mcK_{\cs} = \mcO( \delta -(d_1 + 2), -(d_2 + 2) ) $,
 where $ \mcK_{\cs} = \mcO(\delta-2,-2) $ is the canonical sheaf of $ \cs $,
 $ d_2 \geq 1 $ and:
 $ d_1 \geq 1 $ for $ \delta = 0 $,
 $ d_1 \geq 0 $ for $ \delta = 1 $ and
 $ d_1 \geq 2 $ for $ \delta \geq 2 $.

 Then
 \[
  h^0(S_{\delta}, \Theta_{S_{\delta}}\otimes \mcE) = 0 = h^1(S_{\delta}, \Theta_{S_{\delta}}\otimes \mcE).
 \]

\end{lemma}
\begin{proof}
From the exact sequence \eqref{exact}, we obtain the exact sequence
  \begin{equation*} 
0\rightarrow \mcO( -(d_1+2), -d_2) \rightarrow
 \Theta_{S_{\delta}}\otimes \mcE\xrightarrow{d\psi\otimes 1} \mcO(\delta-d_1,-(d_2+2)) \rightarrow 0.
\end{equation*}
Considering its associated long exact sequence we deduce that it suffices to prove
 \begin{align}\label{pd31}
   h^0(S_{\delta},\mcO( -(d_1+2), -d_2) ) & = 0 = h^0(S_{\delta},\mcO(\delta-d_1,-(d_2+2))), \text{ and } \\ \label{pd32}
   h^1(S_{\delta},\mcO( -(d_1+2), -d_2) ) & = 0 = h^1(S_{\delta},\mcO(\delta-d_1,-(d_2+2))),
 \end{align}
 to get the desired equalities in the statement. For a start, both equations in \eqref{pd31} hold
 for every $ \delta \geq 0 $ \textit{and for every } $ d_1 \in \ZZ $
 because $ -(d_2+2) < -d_2 < 0 $.

In order to prove the equalities in \eqref{pd32}, we recall from \cite[Proposition 2.3]{Laface} that
  $ h^1(S_{\delta},\mcO(a, b)) = 0 $, for integer numbers $ a\geq 0 $, $ b\geq -1 $ and $ \delta \geq 0 $ .
  Therefore, by Serre duality we have
  \begin{equation}\label{Laf&Serre}
   h^1(S_{\delta},\mcO(\delta -(a+2), -(b+2))) = 0, \text{ for }  a\geq 0, b\geq -1 \text{ and } \delta \geq 0.
  \end{equation}

Substitute the value $ b = d_2 - 2 \geq -1 $ in \eqref{Laf&Serre} to get
  \begin{equation}\label{Laf&Serrepd321}
   h^1(S_{\delta},\mcO(\delta -(a+2), -d_2)) = 0, \text{ for }  a\geq 0, d_2\geq 1 \text{ and } \delta \geq 0.
  \end{equation}
Then, \eqref{Laf&Serrepd321} proves the first equality in \eqref{pd32}.
In fact it holds, for $ \delta = 0 $,
for any value of $ a = d_1 \geq 0 $; for $ \delta = 1 $, for any value of $ a = d_1 + 1 \geq 0 $ (that is,
for any value of $ d_1 \geq -1 $) and finally, for $ \delta \geq 2 $, for any value of
$ a = d_1 + \delta \geq 0 $ (that is, for any value of $ d_1 \geq -\delta $).

For the proof of the second equality in \eqref{pd32}, we replace the value $ b = d_2 \geq 1 $
in \eqref{Laf&Serre} to get
  \begin{equation}\label{Laf&Serrepd322}
   h^1(S_{\delta},\mcO(\delta-(a+2),-(d_2+2))) = 0, \text{ for }  a\geq 0, d_2\geq 1 \text{ and } \delta \geq 0.
  \end{equation}
According to the restrictions on the values of $ d_1 $, we see that the substitution $ a = d_1 - 2 \geq 0 $ in \eqref{Laf&Serrepd322} proves the second equality in \eqref{pd32}, \textit{except} for the following cases:
  \begin{align}\label{excpd320}
    h^1(S_{0},\mcO(-1,-(d_2+2))) & = 0, \; d_2 \geq 1, \text{ and } \\ \label{excpd321}
    h^1(S_{1},\mcO(d_1,-(d_2+2))) & = 0, \; d_2 \geq 1, d_1 = 0, 1.
  \end{align}
The proof of these equalities follows from \cite[Proposition 5.3]{Gomez-Mont}. Indeed, the invertible sheaves in
\eqref{excpd320} and \eqref{excpd321} lie in the regions given by Proposition \ref{isolated}
($ d_1\geq 0 $ and $ d_2 \geq 0 $ for $\delta = 0 $, and $ d_1 \geq -1 $ and $ d_2 \geq 0 $ for
$\delta \geq 1 $) where foliations with tangent sheaf $ \mcL = \mcO(-d_1, -d_2) $ may have isolated
singularities. None of them corresponds to the ruling $ \mcO(\tau) = \mcO(-\delta, 2 ) $ nor to
the sheaf $ \mcO( d, 0) $ associated to a Riccati foliation. Moreover, none of the invertible sheaves
from \eqref{excpd320} belong to the exceptional cases  described in \cite[Proposition 5.3, (1)]{Gomez-Mont}
which correspond to $ \delta = 0 $ and $ \mcO(-a, -b) = \mcO(0, -b) $. This finishes the proof of \eqref{excpd320}.

Finally, using \eqref{CoC}, we see that the exceptional cases in \cite[Proposition 5.3, (2)]{Gomez-Mont}
correspond to $ \delta = 1 $ and
  \begin{equation}\label{except21}
     \mcO( -d_1, -d_2 ) = \mcO( n(n-1)/2 + 1, -n ), \text{ with } n \geq 2.
  \end{equation}
Then it is clear that none of the invertible sheaves in \eqref{excpd321} has the form \eqref{except21}. This proves \eqref{excpd321} and the proof is over.
\end{proof}


Let $ [s] \in \PP H^{0}( S_{\delta}, \cts \otimes \mcL^{*}) $
be a foliation with isolated singularities and singular scheme $ Z $ with sheaf of ideals $ I_Z $.
In view of Corollary \ref{manyfols}, there exist other foliations $ [s^{\prime}] \in \mathrm{Fol}(\mcL,\cs) \simeq
 \PP H^{0}( S_{\delta}, \cts \otimes \mcL^{*})$
with the same singular scheme. We seek for them through the following construction. Consider the bundle $ E = T\cs \otimes L^* $. Then its dual
 $ E^* = (T\cs \otimes L^*)^* \simeq T^*\cs \otimes L $
 and $ \bigwedge^2 E^* = \bigwedge^2 (T^*\cs \otimes L)
 \simeq \bigwedge^2 (T^*\cs)\otimes L^{\otimes 2}
  \simeq \mcK_{\cs}\otimes L^{\otimes 2}$. Hence the Koszul resolution of $ Z $ (see \cite{Cam-Oli})
  may be written as
 \begin{equation}\label{KszlZ}
      0 \longrightarrow \bigwedge^2 \Omega_{\cs}^1 \otimes \mcL^{\otimes 2} \stackrel{ \iota_{s} }\longrightarrow
       \Omega_{\cs}^1 \otimes \mcL \stackrel{ \iota_{s} }{\longrightarrow}
        I_Z \longrightarrow 0,
 \end{equation}
 where the maps $ \iota_{s} $ are contraction-by (or evaluation-at) $ s $.
 The tensor product of \eqref{KszlZ} with $ \cts \otimes \mcL^*  $ gives the exact sequence
 \[
 0 \longrightarrow ( \bigwedge^2 \Omega_{\cs}^1 )\otimes \cts \otimes \mcL^* \stackrel{\iota_{s}\otimes 1}\longrightarrow
       \Omega_{\cs}^1\otimes \cts \stackrel{\iota_{s}\otimes 1}{\longrightarrow}
        \cts \otimes \mcL^* \otimes I_Z  \longrightarrow 0,
 \]
where
 $ \Omega_{\cs}^1\otimes \cts \simeq \mathrm{Hom}_{ \mcO} (\cts, \cts) $.
Letting $ \mcE = \mcL \otimes \mcK_{\cs} $, the sequence above may be rewritten as
 \begin{equation}\label{twistedKszlZ}
  0 \longrightarrow \cts \otimes \mcE \stackrel{\iota_{s}\otimes 1}\longrightarrow
   \mathrm{Hom}_{ \mcO} (\cts, \cts) \stackrel{\iota_{s}\otimes 1}\longrightarrow
    \cts \otimes \mcL^* \otimes I_Z  \longrightarrow 0,
 \end{equation}
 with associated long exact sequence given by
\begin{equation}\label{LeCtKszlZ}
 \begin{split}
    0 & \rightarrow H^0(\cs, \cts \otimes \mcE)\xrightarrow{{\iota_{s}\otimes 1}^0}
       H^0(\cs, \mathrm{Hom}_{ \mcO} (\cts, \cts) ) \xrightarrow{{\iota_{s}\otimes 1}^0}
        H^0(\cs, \cts \otimes \mcL^* \otimes I_Z ) \xrightarrow{\delta^0}  \\
      &      \xrightarrow{\delta^0}
       H^1(\cs, \cts \otimes \mcE) \xrightarrow{{\iota_{s}\otimes 1}^1}
        H^1(\cs, \mathrm{Hom}_{ \mcO} (\cts, \cts) ) \xrightarrow{{\iota_{s}\otimes 1}^1}
         H^1(\cs, \cts \otimes \mcL^* \otimes I_Z ) \xrightarrow{\delta^1}
       \cdots .
 \end{split}
\end{equation}
 Finally, notice that $ H^0(\cs, \cts \otimes \mcL^* \otimes I_Z ) $ consists of  those global sections
 in $ H^0(\cs, \cts \otimes \mcL^* ) $ that vanish at $ Z $ and that the effect of the map
 $ \iota_{s}\otimes 1^0 $ in \eqref{LeCtKszlZ} 
 on a global endomorphism $ \Phi $ is $ \iota_{s}\otimes 1^0( \Phi ) = \Phi( s ) $.
 Then, \eqref{LeCtKszlZ} shows that, \textit{every section $ s^{\prime} $ that vanishes on
 $ Z $ is of the form $ s^{\prime} = \Phi( s ) $ for some endomorphism $ \Phi $ if and only
 if the map $ \iota_{s}\otimes 1^0 $ is surjective, } and this is the case \textit{if
 $ h^1(S_{\delta}, \cts \otimes \mcE) = 0 $.}
 This conclusion, together with Lemma \ref{Lemma 3}, gives our main result:

 \begin{theorem}\label{mainT}
 Let $ \mcF = [s] \in \mathrm{Fol}(\mcL,S_{\delta}) \simeq \PP\text{H}^0(\cs, \Theta_{S_{\delta}}\otimes \mcL^*) $
 be a foliation on $ \cs $
 where $ \mcL = \mcO( -d_1, -d_2 ) $ satisfies that
 $ d_2 \geq 1 $, and
 $ d_1 \geq 1 $ if $ \delta = 0 $;
 $ d_1 \geq 0 $ if $ \delta = 1 $, and
 $ d_1 \geq 2 $ if $ \delta \geq 2 $.

 Assume that $ [s] $ has isolated singularities, let $ Z $ be its singular scheme and consider any other
 section $ s^{\prime} \in H^0(\cs, \Theta_{S_{\delta}}\otimes \mcL^*)$ with the same singular
 scheme  $ Z $ as $ [s] $, then there exists a global invertible endomorphism $ \Phi $ of $TS_{\delta} $
 such that $ s^{\prime} = \Phi( s ) $. Moreover, if the affine $1$-form
 $$ \Omega = A_0\;dX_0+A_1\; dX_1+B_0\; dY_0+ B_1\;dY_1 $$ represents the section $ s $, then any
 section $ s^{\prime} = \Phi( s ) $ is represented by an affine $1$-form $ \Omega^{\prime} $ where
 \begin{equation*}\label{thm2}
   \Omega^{\prime} =
    \begin{cases}
     a\cdot(A_0\;dX_0+A_1\; dX_1) + d\cdot(B_0\; dY_0+ B_1\;dY_1), & \mbox{$ a, d\in \CC^* $ if $\delta=0$ }  \\
      a\cdot\Omega, & \mbox{$ a\in\CC^* $, if $\delta = 1 $ } \\
     a\cdot\Omega - Y_1 C(X_1,X_2) B_0 \cdot (X_1 dX_0 - X_0 dX_1) ,
      & \mbox{$ a\in\CC^*, C \in H^0(\PP^1, \mcO_{\PP^1}( \delta -2 )) $, if $\delta\geq 2$.}
    \end{cases}
 \end{equation*}
 It follows in particular that if $ \delta = 1 $, then $ \mcF = [s] $ is uniquely determined by $ Z $, in the
 sense that $ \mcF $ is the unique foliation with singular scheme $ Z $.
 \end{theorem}
\begin{proof}
  The first statement follows  from Lemma \ref{Lemma 3} and the construction described after its proof;
  the second one (the one containing the displayed equation),  from Corollary \ref{manyfols} and,
  the last one, from the equation in the middle of the displayed equality for $\Omega'$.
\end{proof}

\end{document}